\documentclass[a4paper,reqno,10pt]{amsart}

\usepackage[utf8]{inputenc}
\usepackage{color, todonotes}
\usepackage{amsmath,amssymb,amsthm,bbm,amsxtra}
\usepackage{enumitem}
\usepackage[xcolor,pdftex,leftbars]{changebar}
\cbcolor{blue}

\usepackage[linktocpage,bookmarksopen,bookmarksnumbered,
      pdfauthor={Paul Buterus, Holger Sambale},
      pdfkeywords={KEYWORDS}]{hyperref}

\theoremstyle{definition} 
\newtheorem{theorem}{Theorem} 
\newtheorem{lemma}[theorem]{Lemma}
\newtheorem{proposition}[theorem]{Proposition}

\theoremstyle{definition}

\theoremstyle{remark}

\newcommand{\PP}{\mathbb{P}}
\newcommand{\EW}{\mathbb{E}}
\newcommand{\dif}{\mathrm{d}}
\newcommand{\RE}{\mathrm{Re}}
\newcommand{\eul}{\mathrm{e}}

\title{Some notes on moment inequalities for heavy-tailed distributions}
\author{Paul Buterus}
\address{Paul Buterus, Germany}
\email{buterus@math.uni-bielefeld.de}
\author{Holger Sambale}
\address{Holger Sambale, Faculty of Mathematics, Ruhr University Bochum, Germany}
\email{holger.sambale@rub.de}

\thanks{H.\,S.\ was supported by the Deutsche Forschungsgemeinschaft (DFG) via CRC 1283 \emph{Taming uncertainty and profiting from randomness and low regularity in analysis, stochastics and their applications}.}

\date{\today}

\begin{document}

\begin{abstract}
We investigate the relation between moments and tails of heavy-tailed (in particular, Pareto-type) distributions. We also discuss the sharpness of our results in a number of examples under certain regularity conditions like log-convexity. Moreover, we derive concentration bounds for polynomial chaos of any order $d$.
\end{abstract}

\maketitle

\section{Introduction}

In classical situations, the family of $L^p$ norms of a random variable $X$ contains a wealth of information about the tails of $X$. For sub-Gaussian random variables, for instance, it is standard knowledge that the growth of the $L^p$ norms characterizes the tail behavior, i.\,e., the property $\PP(|X| \ge t) \le 2\eul^{-t^2/c^2}$ for some $c > 0$ and any $t \ge 0$ is equivalent to
\[
\lVert X \rVert_{L^p} := (\EW [|X|^p])^{1/p} \le Cp^{1/2}
\]
for any $p \ge 1$, where $C>0$ is some constant only depending on $c$, cf.\ e.g.\ \cite[Proposition 2.5.2]{vershynin:2018}. Analogous results remain valid if the random variable under consideration has slightly heavier tails which are no longer sub-Gaussian but still decay exponentially, i.\,e., $\PP(|X| \ge t) \le 2\eul^{-t^\alpha/c^\alpha}$ for some $\alpha \in (0,2)$ and any $t \ge 0$, cf.\ \cite[Proposition 5.2]{sambale:2023}.

Against this context, one may ask for similar properties of heavy-tailed random variables, i.\,e., random variables with tails of the form
\[
\PP(|X| \ge t) \le Ct^{-\alpha}
\]
for some $\alpha > 0$, any $t \ge c$ and suitable $C,c > 0$. A classical example are Pareto-type random variables. For any $\alpha,b > 0$, a Pareto $(\alpha,b)$ variable $X$, written $X \sim \operatorname{Par}(\alpha,b)$, is a random variable with tails $\PP(X \ge t) = (b/t)^\alpha$ for any $t \ge b$ and density function $f(t) = \alpha b^\alpha/t^{\alpha+1}$ for $t \ge b$. Moreover, by $\operatorname{Par}_\mathrm{s}(\alpha,b)$ we refer to symmetrized Pareto $(\alpha,b)$ variables, which have density $f_\mathrm{s}(t) = \alpha b^\alpha/(2|t|^{\alpha+1})$ for any $|t| \ge b$.

If $X \sim \operatorname{Par}(\alpha,b)$, then $\lVert X \rVert_{L^p} < \infty$ iff $p < \alpha$. More precisely,
\begin{equation}\label{Paretomoments}
\EW [X^p] = b^p\frac{\alpha}{\alpha-p}
\end{equation}
for any $p < \alpha$. In particular, applying Chebyshev's inequality for any $p<\alpha$ and optimizing in $p$ leads to the bound
\begin{equation}
\label{eq:log-error-tails}
\PP(X \ge t) \le \eul \alpha\log(t/b)(b/t)^\alpha
\end{equation}
for any $t \ge b\eul^{1/\alpha}$, which is attained at $p=\alpha-1/\log(t/b)$. Consequently, this approach results in a logarithmic error factor. Moreover, we cannot recover the moments \eqref{Paretomoments} from \eqref{eq:log-error-tails}: indeed, it is easy to check for a random variable $Y$ with tails as in \eqref{eq:log-error-tails}, it holds that there are constants $0 < c_{b,\alpha} < C_{b,\alpha} < \infty$ depending on $b$ and $\alpha$ only such that
\begin{equation}\label{einsweiter}
c_{b,\alpha} (\alpha-p)^{-2} \le \EW [Y^p] \le C_{b,\alpha}(\alpha-p)^{-2}
\end{equation}
as $p \to \alpha$, which corresponds to the $p$-th moment of the product of two independent Pareto variables (up to constant). In some sense, this phenomenon of having access to the moments up to some order $\alpha$ and involving a sort of ``error'' in terms of the distribution is also related to classical topics like the truncated Hamburger moment problem, cf.\ e.\,g.\ \cite[Corollary 2.5.4]{akhiezer:1965}.

In this note, we study the connections between \eqref{Paretomoments} and \eqref{eq:log-error-tails} in more detail. We provide equivalent characterizations of the moment behavior of Pareto-type distributions. Moreover, we construct examples of random variables which demonstrate that the logarithmic error which appears in \eqref{eq:log-error-tails} cannot be avoided in general (in fact, it turns out to be almost optimal). We also demonstrate that the error can be avoided if additional assumptions on the tail behaviour are imposed, which for instance involve log-convexity.

In Section 3, these observations are complemented by moment and concentration inequalities for polynomial chaos in heavy-tailed random variables, i.\,e., functionals of the form $\sum_{i_1,\ldots,i_d} a_{i_1,\ldots,i_d} X_{i_1} \cdots X_{i_d}$ for real-valued coefficients $a_{i_1,\ldots,i_d}$. For $d=2$, we compare our bounds to the famous Hanson--Wright inequality for sub-Gaussian random variables, noting that they may be regarded as preliminary results in this direction.

\section{Moments and tails of heavy-tailed distributions}

\subsection{Moment inequalities}

Let $X$ be a random variable such that the growth of its $p$-th (fractional absolute) moments is at most of Pareto type, i.\,e., $\EW[|X|^p] \le C(\alpha-p)^{-1}$ for any real $p < \alpha$ and a suitable constant $C > 0$. As mentioned in the introduction, in the sub-Gaussian case the corresponding property $\EW[|X|^p] \le Cp^{p/2}$ has a number of well-known reformulations involving the tails, the moment-generating function or the sub-Gaussian norm of $X$. Clearly, and in view of the discussion in the introduction, none of these properties can have a direct ``heavy-tailed'' analogue. Instead, we may reformulate Pareto-type growth of the moments in terms of truncations of $X$ as carried out in the following proposition.

\begin{proposition}
 \label{prop_charac}
 Let $\alpha >0$ and $X$ be a non-negative random variable. Then, the following properties are equivalent:
 \begin{enumerate}[label=(\alph*)]
    \item The moments of $X$ satisfy
          \begin{equation*}
               \EW[X^p] \leq \frac{C_1}{\alpha - p}
          \end{equation*}
          for all $p \in (0,\alpha)$.
    \item The truncated $\alpha$-th moment satisfies
          \begin{equation*}
                \EW[ X^\alpha \mathbbm{1}_{\{ X \leq r\}} ] \leq C_2 \log{r}
          \end{equation*}
          for any $r\ge\mathrm{e}$.
    \item  We have
          \begin{equation*}
              \int_0^r y^{\alpha-1} \PP(X\geq y) \, \dif y \leq C_3 \log{r}
          \end{equation*}
          for any $r \ge \mathrm{e}$.
    \item  We have
          \begin{equation*}
              \EW[X^\alpha \exp(-sX)] \leq C_4 \lvert \log{s} \rvert 
          \end{equation*}
          for any $s \in (0,\mathrm{e}^{-1}]$.
 \end{enumerate}
 Here, the constants $C_i > 0$, $i=1,\ldots,4$, only differ by $\alpha$-dependent factors.
\end{proposition}

Explicit relations between the constants $C_i$ can be found in the proofs. Note that in parts (b) and (c), one may replace the condition $r \ge \mathrm{e}$ by $r \ge r^\ast$ for any fixed $r^\ast > 1$ at the cost of modifying the constants. Similarly, (d) can be extended to $s \in (0,s^\ast]$ for any fixed $s^\ast < 1$.

To prove Proposition \ref{prop_charac}, we start with a technical lemma, specialized for $\alpha=2$, which yields characterizations of Pareto-type growth of the moments of a random variable $X$ in terms of its Laplace transform and its characteristic function, respectively.

\begin{lemma}
 \label{lemma:charac_moments}
 Let $X$ be a non-negative random variable and $C>0$ a fixed constant which may depend on $X$. Then the following statements are equivalent. 
 \begin{enumerate}[label=(\alph*)]
    \item For all $s \in (0,1)$ the moments of $X$ satisfy
          \begin{equation*}
               \EW[X^{1+s}] \leq \frac{C}{1 - s}.
          \end{equation*}
    \item $X$ is integrable and the Laplace transform $L(t) := L_X(t) := \mathbb{E}[\eul^{-tX}]$ of $X$  satisfies
          \begin{equation*}
              \int_0^\infty \frac{\EW[X] + L'(u)}{u^{1+s}} \, \dif u \leq C \frac{\Gamma(1-s)}{s (1-s)}
          \end{equation*}
          for all $s\in (0,1)$.
    \item The characteristic function $\varphi(t) := \varphi_X(t) := \mathbb{E}[\eul^{itX}]$ of $X$ satisfies
          \begin{equation*}
              \int_0^\infty \frac{1 - \RE(\varphi(u))}{u^{2+s}} \, \dif u  \leq C\frac{\sin(\pi s/2) \Gamma(1-s)}{s(s+1)(1-s)}
          \end{equation*}
          for all $s \in (0,1)$.
    \end{enumerate}
\end{lemma}

We shall need the equivalence of (a) and (c) in the sequel. In view of \eqref{einsweiter} and the discussion of polynomial chaos in Section 3, note that similar arguments as used in Lemma \ref{lemma:charac_moments} can also be applied to singularities of the form $(\alpha-s)^{-h}$, where $h>0$. Lemma \ref{lemma:charac_moments} is essentially an application of identities for fractional absolute moments of heavy-tailed distributions from \cite{matsui-pawlas:2016}. For the sake of completeness, we provide some details in the proof.

\begin{proof}
  First note that $L$ is continuously differentiable with $L'(t) = - \EW[X \exp(-tX)]$ and $L'(0) = -\EW[X]$ as long as $X$ is integrable. Therefore, as shown in \cite[Lemma 1.1]{matsui-pawlas:2016}, the identity
  \begin{equation*}
 \frac{\Gamma(1-s)}{s} x^s = \int_0^\infty \frac{1-\exp(-xu)}{u^{s+1}} \, \dif u,
\end{equation*}
valid for $s \in (0,1)$ and $x>0$, implies that
   \begin{equation*}
         \EW[X^{1+s}] = \frac{s}{\Gamma(1-s)} \int_0^\infty \frac{L'(u)-L'(0)}{u^{s+1}} \, \dif u
   \end{equation*}
   for any $s \in (0,1)$, so that the equivalence of (a) and (b) directly follows.

   Similarly, by formula (7) in \cite[Theorem 11.4.3]{kawata:1972},
   \begin{equation*}
  \int_0^\infty \frac{1-\cos(xu)}{u^{1+\beta}} \, \dif u = \frac{\Gamma(2-\beta)}{\beta} \frac{\sin(\pi (1-\beta)/2)}{1-\beta} |x|^\beta
\end{equation*}
for $\beta \in (0,2)$ and $x \in \mathbb{R}$. Choosing $\beta=s+1$ and applying Fubini's theorem, it was shown in \cite[Lemma 1.3 (2)]{matsui-pawlas:2016} that
   \begin{equation*}
      \EW[X^{1+s}] = \frac{s (s+1)}{\sin(\pi s/2) \Gamma(1-s)} \int_0^\infty \frac{1-\RE(\varphi(u))}{u^{2+s}} \, \dif u
   \end{equation*}
   for any $s \in (0,1)$, which establishes the equivalence of (a) and (c).
\end{proof}

\begin{proof}[Proof of Proposition \ref{prop_charac}]
 To prove that (a) implies (b), we temporarily assume that $\alpha = 2$ by replacing $X$ by the random variable $\tilde{X} := X^{\alpha/2}$ and $C_1$ by $\tilde{C}_1 := 2C_1/\alpha$. Using part (c) of Lemma \ref{lemma:charac_moments}, we then obtain
   \begin{equation*}
    \tilde{C}_1 \frac{\sin(\pi s/2) \Gamma(1-s)}{s(s+1)(1-s)} \geq \int_0^\infty \frac{1 - \RE(\varphi_{\tilde{X}}(u))}{u^{2+s}} \, \dif u \geq \int_0^{\pi/\tilde{r}} \frac{1 - \RE(\varphi_{\tilde{X}}(u))}{u^{2+s}} \, \dif u
   \end{equation*}
    for any $\tilde{r} > 0$. The last integral can be estimated further by
   \begin{align*}
        \int_0^{\pi/\tilde{r}} \frac{1 - \RE(\varphi_{\tilde{X}}(u))}{u^{2+s}} \, \dif u
        &\geq 2 \EW \Big[ \mathbbm{1}_{\{ \tilde{X} \leq \tilde{r}\}} \int_0^{\pi/\tilde{r}} \frac{\sin(\tilde{X}u/2)^2}{u^{2+s}} \, \dif u\Big] \\
        &\geq \frac{2}{\pi^2} \EW[ \mathbbm{1}_{\{ \tilde{X} \leq \tilde{r}\}} \tilde{X}^2] \int_0^{\pi/\tilde{r}} u^{-s} \, \dif u \\
        &\geq \frac{2}{\pi^{1+s}} \frac{\tilde{r}^{s-1}}{1-s}  \EW[ \mathbbm{1}_{\{ \tilde{X} \leq \tilde{r}\}} \tilde{X}^2],
   \end{align*}
   since $1-\cos(x)=2\sin(x/2)^2$ for any $x$ and $\sin(x) \geq 2x/\pi$ on $x \in [0,\pi/2]$ (use e.\,g.\ that $\sin(x)$ is concave on $[0,\pi/2]$). Combining both inequalities yields
   \begin{equation*}
         \frac{\pi^{1+s}}{2} \tilde{C}_1 \frac{\sin(\pi s/2) \Gamma(1-s)}{s(s+1)} \tilde{r}^{1-s} \geq  \EW[ \mathbbm{1}_{\{ \tilde{X} \leq \tilde{r}\}} \tilde{X}^2].
   \end{equation*}
   Next, we optimize this expression in $s$ by taking $s = 1- \log(\tilde{r})^{-1} \in [0,1)$ if $\tilde{r} \ge \mathrm{e}$. Using that $\sin(x) \leq x$ for any $x \ge 0$ as well as $\Gamma(x) = x^{-1} \Gamma(1+x) \leq x^{-1}$ on $x \in (0,1]$, we find that
   \begin{equation*}
       \EW[ \mathbbm{1}_{\{ \tilde{X} \leq \tilde{r}\}} \tilde{X}^2] \leq \frac{\mathrm{e} \, \pi^3}{4} \, \tilde{C}_1 \log(\tilde{r}).
   \end{equation*}
   Now setting $\tilde{r} := r^{\alpha/2}$, it follows that
   \begin{equation*}
       \EW[ \mathbbm{1}_{\{ X \leq r\}} X^\alpha] = \EW[ \mathbbm{1}_{\{ \tilde{X} \leq \tilde{r}\}} \tilde{X}^2] \leq \frac{\mathrm{e} \, \pi^3}{4} \, \tilde{C}_1 \log(\tilde{r}) = \frac{\mathrm{e} \, \pi^3}{4} \, C_1 \log(r)
   \end{equation*}
   provided that $r \ge \eul^{2/\alpha}$. If $\alpha \ge 2$, this yields (b), while if $\alpha < 2$, it remains to note that
   \[
   \EW[ \mathbbm{1}_{\{ X \leq r\}} X^\alpha] \le \EW[ \mathbbm{1}_{\{ X \leq \eul^{2/\alpha}\}} X^\alpha] \le \frac{\mathrm{e} \, \pi^3}{2\alpha} \, C_1 \le \frac{\mathrm{e} \, \pi^3}{2\alpha} \, C_1 \log(r)
   \]
   for any $r \in [\eul,\eul^{2/\alpha}]$.
   
   To show that that (b) implies (a), we replace $X$ by $\hat{X} := X\mathbbm{1}_{\{X \ge \mathrm{e}\}}$. We may then write
   \begin{equation*}
        \EW[\hat{X}^p] = (\alpha-p) \int_1^\infty u^{-(\alpha-p)-1} \EW[\hat{X}^\alpha \mathbbm{1}_{\{\hat{X} \leq u\}}] \, \dif u,
   \end{equation*}
   noting that $\mathbbm{1}_{\{\hat{X} \leq u\}} = 0$ if $u < \mathrm{e}$, and thus conclude that
   \begin{equation*}
       \EW[\hat{X}^p] \leq C_2  (\alpha-p) \int_1^\infty u^{ -(\alpha-p)-1} \log{u} \, \dif u = \frac{C_2}{\alpha-p}
   \end{equation*}
   and hence
   \begin{equation*}
       \EW[X^p] \leq \eul^\alpha + \frac{C_2}{\alpha-p} \le \frac{C_2 + \alpha\eul^\alpha}{\alpha-p}.
   \end{equation*}

   Moreover, noting that
      \begin{equation*}
        \EW[X^\alpha \mathbbm{1}_{\{X \leq r\}}]  = \alpha \int_0^r y^{\alpha-1} \PP(y \leq X \leq r) \, \dif y \leq \alpha \int_0^r y^{\alpha-1} \PP(X \geq y) \, \dif y,
   \end{equation*}
   (c) trivially implies (b) with $C_2 := \alpha C_3$. On the other hand, using that
   \begin{equation*}
       \int_0^r y^{\alpha-1} \PP(X \ge y) \, \dif y = \int_0^r y^{\alpha-1} \PP(y \le X \le r) \, \dif y + \frac{r^\alpha}{\alpha} \PP(X \ge r),
   \end{equation*}
   the first term is bounded $C_2\alpha^{-1}\log(r)$ due to (b), and by (a) and Chebyshev's inequality similarly as in \eqref{eq:log-error-tails}, the second term is bounded by $C_1 \mathrm{e} \alpha^{-1} \log(r)$. In particular, (c) follows with $C_3 := (C_1 \mathrm{e} + C_2)/\alpha$.
   
   To prove the equivalence of (b) and (d), we observe that
   \begin{align*}
         \EW[X^\alpha \exp(-sX)] &= s \EW[X^\alpha \int_X^\infty \exp(-s t) \, \dif t]    = s \int_0^\infty \eul^{-st} \EW[ X^\alpha \mathbbm{1}_{\{X \leq t\}}] \, \dif t \\
         &\leq \eul^\alpha+ C_2 s \int_1^\infty \eul^{-st} \log(t) \, \dif t  = \eul^\alpha + C_2 \int_s^\infty \frac{\exp(-u)}{u} \dif u
   \end{align*}
   by splitting the integral into the parts $t \le \mathrm{e}$ and $t \ge \mathrm{e} \ge 1$ and integration by parts in the last step. Splitting the integral on the right-hand side at $u =1$ then yields
   \begin{equation*}
        \EW[X^\alpha \exp(-sX)] \le \eul^\alpha + C_2(\eul^{-1} + |\log{s}|) \le (\eul^\alpha + C_2(\eul^{-1} + 1))|\log{s}|
   \end{equation*}
   for any $s \le \mathrm{e}^{-1}$. On the other hand, we have
   \begin{equation*}
        \EW[X^\alpha \exp(-sX)] \geq \eul^{-1} \EW[X^\alpha \mathbbm{1}_{\{X \leq 1/s\}}],
   \end{equation*}
   which implies (b) with $C_2 := \mathrm{e} C_4$.
\end{proof}

In view of the discussion in the introduction, and as we will confirm in the course of this note, the properties listed in Proposition \ref{prop_charac} cannot be complemented by a statement about the tails in general. However, it turns out that they are equivalent to a Pareto-type tail behaviour if we impose additional ``structural'' assumptions on the tails.

To this end, recall that a function $f \colon [b,\infty) \to \mathbb{R}$ for some $b \ge 0$ is called log-convex if $\log (f)$ is convex on $[b,\infty)$. In the following, we are interested in random variables with log-convex tails. If $X$ is non-negative, this means that its tails $\PP(X \ge t)$ are a log-convex function of $t \in [b,\infty)$ with $b := \mathrm{ess }\inf X$. More generally, a symmetric real-valued random variable $X$ has log-convex tails if $|X|$ has log-convex tails. In the same way, one may also define random variables with log-concave tails.

Clearly, Pareto random variables have log-convex tails. Note that sometimes, in the definition of (say, non-negative) random variables $X$ with log-convex tails one requires the tails of $X$ to be log-convex on the full non-negative half-axis $[0,\infty)$, i.\,e.\ not only restricted to $t \ge b = \mathrm{ess }\inf X$. To see that the two definitions essentially lead to the same concept, one may replace $X$ by $X-b$ if necessary. For instance, if $X \sim \operatorname{Par}(\alpha,b)$, it holds that $2^{-\alpha}\mathbb{P}(X \ge t) \le \mathbb{P}(X-b\ge t) \le \PP(X \ge t)$ for all $t\ge0$. In particular, similar relations also hold for the $L^p$ norms of $X$ and $X-b$, noting that $2^{-\alpha} \EW[X^p] \le \EW[(X-b)^p] \le \EW[X^p]$ for any $p > 0$.

\begin{proposition}\label{prop_3}
 If $t^\alpha \PP(X \geq t)$ is $\log$-convex on $t \in [b,\infty)$ with $b := \mathrm{ess }\inf X$, then the characterizations of Proposition \ref{prop_charac} are also equivalent to the tail bound
 \begin{equation*}
     \PP(X \geq t) \leq \frac{C_5}{t^\alpha}
 \end{equation*}
 for $t \ge\mathrm{e}$, where the constant $C_5 > 0$ only differs from the constants $C_1$ to $C_4$ by $\alpha$-dependent factors. The same statement holds if $t^\alpha \PP(X \geq t)$ is $\log$-concave and bounded from below by some $\delta > 0$.
\end{proposition}

In particular, the conditions of Proposition \ref{prop_3} hold for $X \sim \mathrm{Par}(\alpha,b)$, so that we get back the correct tail behaviour this way. Assuming that $t^\alpha \PP(X \ge t) \ge \delta$ in the case of log-concavity has technical reasons, but one can easily see that it does not exclude any situations of interest. Indeed, otherwise we must have $t^\alpha \PP(X \ge t) \to 0$ as $t \to \infty$, and by log-concavity the rate of decay must be so fast that $X$ cannot be a heavy-tailed random variable.

\begin{proof} First recall that as already mentioned in the introduction, $\PP(X \geq t) \leq C_5/t^\alpha$ always implies part (a) of Proposition \ref{prop_charac} with $C_1 = \alpha C_5$ without any further assumptions. (To be precise, here we assume that the tail bound holds for any $t \ge 1$, but even if we can only access it for $t \ge \eul$ replacing $C_5$ by $\max(C_5,\eul^\alpha)$ will lead to the same result.)

Conversely, set $h(t) = t^\alpha \PP(X \ge t)$. If $h(t)$ is log-convex, assume part (a) of Proposition \ref{prop_charac}. As in \eqref{eq:log-error-tails} (with $b=1$), we have $h(t) \le C_1 \mathrm{e} \log(t)$ by Chebyshev's inequality and optimization. In particular, $\log(h(t))$ is upper-bounded by the concave function $1 + \log(C_1) + \log\log(t)$. Since $\log(h(t))$ is convex by assumption, this implies that $h(t)$ must be non-increasing. Indeed, by convexity, once there are $t_1 < t_2$ such that $\log(h(t_1)) < \log(h(t_2))$, it follows that all points $(t,\log(h(t)))$ for $t \ge t_2$ must lie above the line defined by $(t_1,\log h(t_1))$ and $(t_2,\log h(t_2))$, which eventually surpasses the graph of the function $1+\log(C_1)+\log\log(t)$. Consequently, for any $t \ge \mathrm{e}$, we have $h(t) \le C_1 \mathrm{e}$.

If $h(t)$ is log-concave, assume part (c) of Proposition \ref{prop_charac}. Fixing any $r \ge \mathrm{e}$, we first consider the case where $h(t)$ is unbounded, i.\,e.\ there exists a monotone sequence $(t_n)_{n \in \mathbb{N}}$ with $t_n \rightarrow \infty$ and $h(t_n) \rightarrow \infty$. Additionally, we can assume that $h(r) \leq h(t_n)$, which is true for large enough $n$. By log-concavity of $h$, we have
\begin{equation*}
  h(u) = h((1-s)r + s t_n ) \geq h(r)^{1-s} h(t_n)^s \geq \min\{h(r),h(t_n)\} = h(r)
\end{equation*}
for all $u =(1-s)r + s t_n$ with some $s \in [0,1]$. Thus, noting that $h(u)/u = u^{\alpha-1}\PP(X \ge u)$ and applying Proposition \ref{prop_charac} (c) it follows that
\begin{equation*}
     C_3 \log(t_n) \geq \int_r^{t_n} \frac{h(u)}{u} \, \dif u \geq h(r) \log(t_n/r).
\end{equation*}
Dividing by $\log(t_n)$ and letting $n \rightarrow \infty$ yields $h(r) \le C_3$. 
 
Now assume $h$ to be bounded by some constant $K>1$. Using the log-concavity again gives
 \begin{equation}\label{Gl1}
  C_3 \log(t) \geq \int_r^t \frac{h(u)}{u} \, \dif u \geq h(r) \Big( \frac{h(r)}{h(t)} \Big)^\frac{r}{t-r} \int_{r}^{t}  \Big( \frac{h(t)}{h(r)} \Big)^\frac{u}{t-r} \, \frac{\dif u}{u}.
\end{equation}
Recall that $h(t) \ge \delta$ for some $\delta > 0$ by assumption. Hence, $\delta \leq h(t) \leq K$, from where we deduce that
\begin{equation}\label{eq:limVerh}
  \lim_{t \rightarrow \infty} \Big( \frac{h(r)}{h(t)} \Big)^\frac{r}{t-r} =1.
\end{equation}
Writing $\beta(t) = \log( h(t)/h(r))$ and performing a change of variables, we have
\begin{equation}\label{Gl2}
  \int_{r}^{t}  \Big( \frac{h(t)}{h(r)} \Big)^\frac{u}{t-r} \, \frac{\dif u}{u} = \int_{\frac{r}{t-r}}^{\frac{t}{t-r}} \exp(\beta(t) u)  \, \frac{\dif u}{u}.
\end{equation}
Using the boundedness of $\beta(t)$, we find by partial integration that
\begin{align}
\lim_{t \rightarrow \infty} \frac{1}{\log(t)} \int_{\frac{r}{t-r}}^{\frac{t}{t-r}} \exp(\beta(t) u)  \, \frac{\dif u}{u}
&= \lim_{t\to\infty}\Big(\frac{\log(u) \exp(\beta(t) u)}{\log(t)} \Big|_{\frac{r}{t-r}}^{\frac{t}{t-r}}\notag\\
&\hspace{-2cm}- \frac{\beta(t)}{\log(t)} \int_\frac{r}{t-r}^{\frac{t}{t-r}} \log(u) \exp(\beta(t) u) \, \dif u\Big) = 1,\label{Gl3}
\end{align}
where we have used \eqref{eq:limVerh}. Dividing both sides of \eqref{Gl1} by $\log (t)$ and recalling \eqref{Gl2}, we obtain
\[
C_3 \geq h(r) \Big( \frac{h(r)}{h(t)} \Big)^\frac{r}{t-r} \frac{1}{\log(t)} \int_{\frac{r}{t-r}}^{\frac{t}{t-r}} \exp(\beta(t) u)  \, \frac{\dif u}{u}.
\]
Now letting $t \to \infty$ and using \eqref{eq:limVerh} and \eqref{Gl3}, it follows that $h(r) \le C_3$.
\end{proof}

\subsection{Examples}
As we have seen in Proposition \ref{prop_3}, in certain situations one may avoid the logarithmic error \eqref{eq:log-error-tails} which appears when combing Pareto-type moment bounds with the Chebyshev inequality. Next, we will present a number of distributions on the real line or, equivalently, random variables, which demonstrate that in general, this is not possible. Here we also consider situations where the tails satisfy certain regularity properties.

We first recall a number of basic definitions and results. A function $L \colon (0,\infty) \to (0,\infty)$ is called slowly varying if for any $a > 0$,
\[
\lim_{t \to \infty} \frac{L(at)}{L(t)} = 1.
\]
Moreover, a function $g \colon (0,\infty) \to (0,\infty)$ is called regularly varying if there exists $\beta \in \mathbb{R}$ such that for any $a > 0$,
\[
\lim_{t \to \infty} \frac{g(at)}{g(t)} = a^\beta.
\]
The parameter $\beta$ is called the index of $g$. By Karamata's characterization theorem, if $g$ is regularly varying of index $\beta$, it has to be of the form $g(t) = t^\beta L(t)$ for some slowly varying function $L$. For a reference, see \cite[Ch.\ 1.3 \& 1.4]{bingham-goldie-teugels:1987}.

\begin{proposition}\label{prop:tailsopt}
    Fix any $\alpha > 0$, and let $h \colon [0,\infty) \to [0,\infty)$ be any increasing function such that $h(t) \to \infty$ as $t \to \infty$ and $\log h(t) < \rho (\alpha\wedge 1)\log\log(t)$ for all $t$ large enough and some $\rho\in(0,1)$. Then, there exists a random variable $X$ satisfying the properties from Proposition \ref{prop_charac} such that
    \begin{equation}\label{eq:Flankenangen}
    \limsup_{t \to \infty} \frac{t^\alpha \mathbb{P}(X \ge t)}{h(t)} \ge 1.
    \end{equation}
    If $\log h(t) = o(\log\log(t))$, $X$ can be chosen such that its tails are a regularly varying function of $t$ with index $-\alpha$.
\end{proposition}

\begin{proof}
  To start, let us consider a general class of functions of the form
  \begin{equation}\label{Startpunkt}
      g(t) := t^{-\alpha} L(t)
  \end{equation}
  with $L(t) := \exp( \int_1^t \varepsilon(s) s^{-1} \, \dif s)$ and a bounded measurable function $\varepsilon$. If $g'(t) \leq 0$ and $g(t) \to 0$ as $t \to \infty$, the function $g$ gives rise to the tails of a random variable on $[1,\infty)$ by setting $\mathbb{P}(X \ge t) := g(t)$, noting that $g(1)=1$. Assuming $\varepsilon$ to be piecewise continuous, the condition $g'(t) \le 0$ holds if and only if $\varepsilon(t) \leq \alpha$, since we have
  \begin{equation*}
       g'(t) = (-\alpha + \varepsilon(t)) t^{-1} g(t).
  \end{equation*}
  In particular, if $\varepsilon(t) \le \beta$ for some $\beta < \alpha$ and all $t \ge 1$ we also have $g(t) \to 0$ (indeed, note that in this case, $0 \leq g(t) \leq t^{-\alpha + \beta}$).
  
  To specify a choice of $\varepsilon$, now define sequences $a_n := \exp(n \exp(n))$ and $b_n := \exp(n \exp(n)+n)$. Moreover, let
  \[
  \gamma(n) := \log(h(b_n))/n = \frac{\log(h(b_n))}{\log\log(b_n)}(1+o(1)).
  \]
  By assumption, for $n$ large enough we have $\gamma(n) \le \rho(\alpha \wedge 1)$. On the other hand, we can also assume that $\gamma(n) \ge 1/n$ for all $n$ without loss of generality. Indeed, if $\gamma(n')\le 1/n'$ for some subsequence $n'$, the sequence $h(b_{n'}) = \exp(n'\gamma(n'))$ is absolutely bounded, which contradicts the assumption of $h(t) \to \infty$ as $t \to \infty$. For $t \in [a_n, a_{n+1})$, set
  \begin{equation}
      \label{eq:eps_bsp1}
      \varepsilon(t) =  \begin{cases} \gamma(n) &  \text{ if } t \in [a_n,b_n) \\ - \gamma(n) & \text{ if } t \in [b_n,b_n \exp(n)) \\
       0 & \text{ if } t \in [b_n \exp(n),a_{n+1}) \end{cases}.
   \end{equation}

   It is easy to see that for $t \in [a_n, a_{n+1})$,
   \[
   \int_1^t \frac{\varepsilon(s)}{s} \, \dif s = \begin{cases}
       \gamma(n)(\log(t)-\log(a_n)) & \text{ if } t \in [a_n,b_n)\\
\gamma(n)(n + \log(b_n) - \log(t)) & \text{ if } t \in [b_n, b_n\exp(n))\\
0 & \text{ if } t \in [b_n \exp(n),a_{n+1})
   \end{cases},
  \]
  and hence,
  \[
   L(t) = \begin{cases}
       \frac{t^{\gamma(n)}}{a_n^{\gamma(n)}} & \text{ if } t \in [a_n,b_n)\\
\exp(n\gamma(n)) \frac{b_n^{\gamma(n)}}{t^{\gamma(n)}} & \text{ if } t \in [b_n, b_n\exp(n))\\
1 & \text{ if } t \in [b_n \exp(n),a_{n+1})
   \end{cases}.
  \]
  In particular, $b_n^{\alpha}g(b_n) = L(b_n) = \exp(n\gamma(n)) = h(b_n)$ by definition. Hence, \eqref{eq:Flankenangen} holds for a random variable $X$ which has tails $\PP(X \ge t) = g(t)$ for any $t \ge 1$.
   
   It remains to check that $X$ satisfies the equivalent properties from Proposition \ref{prop_charac}, of which we will verify (c). To this end, we first consider $r \in [a_{n},b_{n}]$ and note that $\log(r) \geq \log(a_{n}) = n \exp(n) \geq \log(b_{n})/2$. Thus, we may assume that $r= b_{n}$. For any $k \le n$, we get
   \begin{align*}
       \int_{a_k}^{b_k} t^{\alpha-1} \PP(X \geq t) \, \dif t = \int_{a_k}^{b_k} \frac{L(t)}{t} \, \dif t &= a_k^{-\gamma(k)} \int_{a_k}^{b_k} t^{-1+\gamma(k)} \, \dif t \\
       &= \gamma(k)^{-1} \Big(\Big( \frac{b_k}{a_k} \Big)^{\gamma(k)}  -1\Big) \leq \gamma(k)^{-1} \exp(k\gamma(k)).
   \end{align*}
   Summing this up for $k \le n$ gives a contribution of at most
   \[
   \sum_{k=1}^n \gamma(k)^{-1} \exp(k\gamma(k)) \le n \max_{k\le n}\gamma(k)^{-1} \exp(n\gamma(n)) \le C \log(b_{n})
   \]
   for a suitable constant $C > 0$. Here, the first step follows from $\gamma(k)k \le \gamma(n)n$ by definition of $\gamma(k)$ and since $h$ is increasing. To see the second step, note that the slightly stronger inequality $n \max_{k\le n}\gamma(k)^{-1} \exp(n\gamma(n)) \le C\log(a_{n})$ can be rewritten as $\exp(n\gamma(n))/\min_{k\le n}\gamma(k) \le C \exp(n)$ and use that $1/n \le \gamma(n) \le \rho<1$. If $r \in [b_k,\exp(k)b_k]$, then again $\log(r) \geq \log(b_k) \geq \log(\exp(k) b_k)/2$ and thus we may consider $r = \exp(k)b_k$. Here we have
   \begin{equation*}
       \int_{b_k}^{b_k \exp(k)} \frac{L(t)}{t} \, \dif t = \exp( k \gamma(k) ) b_k^{\gamma(k)} \int_{b_k}^{b_k \exp(k)} \frac{1}{t^{1+\gamma(k)}} \, \dif t \leq \gamma(k)^{-1} \exp(k\gamma(k)),
   \end{equation*}
   so that we may argue similarly as above. Finally, if $r \in [b_k \exp(k),a_{k+1}]$,
   \begin{equation*}
       \int_{b_k \exp(k)}^{a_{k+1}} \frac{L(t)}{t} \, \dif t = \int_{b_k \exp(k)}^{a_{k+1}}  \frac{1}{t} \, \dif t \leq \log(a_{k+1}) - \log(b_k).
   \end{equation*}
   Because of $a_k \leq b_k$, summing up and telescoping yields a bound of $\log(r)$ again. In total, we have thus verified condition (c) of Proposition \ref{prop_charac}.

   Finally, note that if $\log h(t) = o(\log\log(t))$, we have $\varepsilon(t) \rightarrow 0$, so that by Karamata's representation theorem (cf.\ \cite[Th.\ 1.3.1]{bingham-goldie-teugels:1987}), $L$ is slowly varying and hence, $g$ is regularly varying with index $-\alpha$ in this case.
\end{proof}

As for the auxiliary function $\gamma(n)$ from the proof of Proposition \ref{prop:tailsopt}, simple examples like $\gamma(n):=n^{\delta-1}$ for any $\delta \in (0,1)$ correspond to $b_n^\alpha\PP(X \ge b_n) = \exp(n^\delta)) \approx \exp((\log\log(b_n))^\delta)$. In fact, for a first understanding it can be instructive to adapt the proof to the case of $\gamma(n) = n^{-1/2}$.

In view of Proposition \ref{prop_3}, one may wonder whether in order to avoid the logarithmic error, it might already suffice to require the tails of $X$ to be log-convex, which would sharpen Proposition \ref{prop_3} where $t^\alpha\PP(X \ge t)$ was assumed to be log-convex. It turns out that this is not true as we demonstrate in the following proposition, whose proof is based on the proof of Proposition \ref{prop:tailsopt} together with a smoothening argument.

\begin{proposition}\label{prop-logcon}
    In the situation of Proposition \ref{prop:tailsopt}, under the assumption that $\log h(t) < \rho(\alpha\wedge 1)\log\log(t)$ for all $t$ large enough and some $\rho \in (0,c_1)$, the random variable $X$ can be chosen such that its tails are log-convex. Here, $c_1 > 0.48$ is some absolute constant.
\end{proposition}

\begin{proof}
  Consider a function $g$ as in \eqref{Startpunkt} and set $\ell(t) := \log (g(t))$. Assuming $\varepsilon(s)$ to be continuously differentiable, one can check in the same way as in the proof of Proposition \ref{prop:tailsopt} that $\ell''(t) \geq 0$ is equivalent to $g''(t) g(t) - (g'(t))^2 \geq 0$ and thus to
  \begin{equation}
      \label{eq:convex_cond_eps}
      \varepsilon'(t) + \frac{\alpha}{t} \geq \frac{\varepsilon(t)}{t}.
  \end{equation}
  Therefore, once \eqref{eq:convex_cond_eps} holds, the corresponding random variable has log-convex tails. By continuity of $\ell'$, the same arguments remain valid if $\varepsilon(t)$ is continuous but only piecewise continuously differentiable (with \eqref{eq:convex_cond_eps} for all continuity points of $\varepsilon'$).
  
  To construct such a function, we write $\Lambda(s) := \min\{2s, 2- 2s\} \in [0,1]$ for $s \in [0,1]$ and replace \eqref{eq:eps_bsp1} by the smoothed version
    \begin{equation*}
      \varepsilon(t) =  \begin{cases} \gamma(n) c_n^{-1} \Lambda( \tfrac{t-a_n}{b_n-a_n})  &  \text{ if } t \in [a_n,b_n) \\ - \gamma(n)  c_n^{-1} \Lambda( \tfrac{t-b_n}{b_n (\eul^n-1)}) & \text{ if } t \in [b_n,b_n \exp(n)) \\
       0 & \text{ if } t \in [b_n \exp(n),a_{n+1})
       \end{cases}
   \end{equation*}
   with correction factor
   \begin{equation*}
       c_n := \int_0^1 \frac{\Lambda(t)}{t+(\eul^n-1)^{-1}} \, \dif t = \int_{a_n}^{b_n} \Lambda( \tfrac{t-a_n}{b_n-a_n}) t^{-1} \, \dif t = \int_{b_n}^{b_n \exp(n)} \Lambda( \tfrac{t-b_n}{b_n (\eul^n-1)}) t^{-1} \, \dif t.
   \end{equation*}
   Note that $c_n$ is monotoneously increasing with
   \[
   c_1 \approx 0.480156,\qquad \lim_{n \rightarrow \infty} c_n = \int_0^1 \frac{\Lambda(t)}{t} \, \dif t = \log(4).
   \]
   In particular, as in the proof of Proposition~\ref{prop:tailsopt} we have that $\varepsilon(t) \le \rho'(\alpha \wedge 1)$ for $\rho':=\rho c_1^{-1}\in(0,1)$, so that $g(t)$ is monotoneously decreasing with $g(t) \to 0$, and $\varepsilon(t) \to 0$ if $\log h(t) = o(\log\log(t))$. Moreover, it also follows that
   \[
   \int_{1}^{b_n} \varepsilon(s)s^{-1} \, \dif s = \gamma(n)n,\qquad \int_{1}^{b_n\exp(n)} \varepsilon(s) s^{-1} \, \dif s = 0,
   \]
   so that the relation $b_n^\alpha g(b_n) = \exp(n\gamma(n)) = h(b_n)$ continues to hold. Furthermore, part (c) of Proposition \ref{prop_charac} can be verified as in the proof of Proposition \ref{prop:tailsopt} by invoking the trivial bound $\Lambda(t) \le 1$.

   It remains to check \eqref{eq:convex_cond_eps}, which is trivial if $\varepsilon'(t) \ge 0$. For $t \in ((b_n-a_n)/2,b_n)$, \eqref{eq:convex_cond_eps} reads
   \[
   \varepsilon(t) = \frac{\gamma(n)}{c_n}\Big(2 - 2 \frac{t-a_n}{b_n-a_n}\Big) \le \alpha - 2\frac{\gamma(n)}{c_n}\frac{t}{b_n-a_n} = \alpha - t\varepsilon'(t),
   \]
   and the inequality in the middle is equivalent to
   \[
   2\frac{\gamma(n)}{c_n} \Big(\frac{1}{\eul^n -1} + 1\Big) \le \alpha.
   \]
   Recalling that $\gamma(n) \le \rho\alpha$ with $\rho < c_1$ and that $c_1/c_n \le 1/2$ for $n$ large enough, this inequality is true for all large enough $n$ (formally, one may replace $\varepsilon(t)$ by $0$ for small values of $n$). Similar arguments also hold in the remaining case of $t \in (b_n, b_n(1+\eul^n)/2)$, which is even simpler since $\varepsilon(t) \le 0$ in this range.
\end{proof}

\section{Tail inequalities for polynomial chaos}

\subsection{Results}
Typical applications of $L^p$ norm inequalities include concentration results of higher order, in particular bounds for polynomials (or polynomial chaos) of some order $d$ in independent random variables $X_1, \ldots, X_n$. Concentration results of this type have been shown in \cite{latala:2006} (polynomials in Gaussian random variables), \cite{adamczak-wolff:2015} (functions of sub-Gaussian random variables) and \cite{goetze-sambale-sinulis:2021} (polynomials in $\alpha$-sub-exponential random variables), for example. All these papers in fact prove bounds on the respective $L^p$ norms, from which concentration bounds are subsequently derived.

We may proceed similarly, using the following generalization of \eqref{eq:log-error-tails}: given a random variable $Z$ which satisfies the moment condition
\[
\lVert Z \rVert_{L^p} \le \gamma \Big(\frac{\alpha}{\alpha - p}\Big)^{\frac{\beta}{p}}
\]
for any $p \in (0,\alpha)$ and some $\gamma,\beta>0$, a standard Chebyshev-type argument shows that
\begin{equation}\label{eq:log-error-tails-gen}
\mathbb{P}(|Z| \ge t) \le (\alpha/\beta)^\beta \mathrm{e}^\beta \log(t/\gamma)^\beta (\gamma/t)^\alpha
\end{equation}
for any $t>\gamma \eul^{\beta/\alpha}$ (here one argues similarly as in \eqref{eq:log-error-tails}, choosing $p = \alpha - \beta/\log(t/\gamma)$).

In the sequel, for a random vector $X=(X_1, \ldots, X_n)$ with independent centered components (i.\,e., $\mathbb{E}[X_i] = 0$ for all $i$), we study functions of type
\begin{equation}\label{multPol}
f_d(X) := f_{d,A}(X) := \sum_{i_1, \ldots, i_d=1}^n a_{i_1 \ldots i_d} X_{i_1} \cdots X_{i_d},
\end{equation}
where $A = (a_{i_1 \ldots i_d}) \in \mathbb{R}^{n^d}$ is a tensor of real coefficients such that $a_{i_1 \ldots i_d} = 0$ whenever $i_j = i_{j'}$ for some $j \ne j'$ (we also say that $A$ has \emph{generalized diagonal} $0$). These polynomials are linear in every component, hence they are also referred to as multilinear (or tetrahedral) polynomials.

For multilinear polynomials in heavy-tailed random variables, we may show the following elementary $L^p$ (and concentration) bound, which shares some similarities with \cite[Theorem 1.5]{goetze-sambale-sinulis:2021} (whose proof it partially mimics). Here and all over the rest of this section, absolute constants with values in $(0,\infty)$ are typically denoted by $C$, constants which depend on some quantities $Q_1,\ldots,Q_m$ only are denoted by $C_{Q_1,\ldots,Q_m}$, and the precise values of the constants may vary from line to line (especially in the course of the proofs).

\begin{proposition}\label{TailsMultPol}
Let $X = (X_1, \ldots, X_n)$ be a random vector with independent centered (i.\,e., $\EW[X_i]=0$) components such that there exist $\alpha > 2$ and $b > 0$ with $\mathbb{P}(|X_i| \ge t) \le (b/t)^\alpha$ for any $t \ge b$ and any $i = 1, \ldots, n$. Let $A = (a_{i_1 \ldots i_d})$ be a tensor with generalized diagonal $0$. Then,
\[
\lVert f_{d,A} \lVert_{L^p} \le C_{d,\alpha} \lVert A \rVert_\mathrm{HS} b^d \Big(\frac{\alpha}{\alpha - p}\Big)^{d/p}
\]
for any $p \in [2,\alpha)$. In particular,
\[
\mathbb{P}(|f_{d,A}(X)| \ge t) \le C_{d,\alpha} \log^d\Big(\frac{t}{\lVert A \rVert_\mathrm{HS}b^d}\Big) \Big(\frac{\lVert A \rVert_\mathrm{HS}b^d}{t}\Big)^\alpha
\]
for all $t \ge C_{d,\alpha} \lVert A \rVert_\mathrm{HS} b^d$.
\end{proposition}

Note that this result can also be derived from \cite[Theorem 1]{kolesko-latala:2015}. However, we present an own proof especially adapted to the random variables under consideration in this note (and thus easier).

\begin{proof}
In view of \eqref{eq:log-error-tails-gen}, it suffices to prove the $L^p$ bound, and moreover, we may clearly assume $b=1$. Let $X^{(1)}, \ldots, X^{(d)}$ be independent copies of the vector $X$ and $(\varepsilon_i^{(j)})$, $i \le n$, $j \le d$, a set of i.i.d.\ Rademacher variables which are independent of the $(X^{(j)})_j$. First, by standard decoupling (cf.\ \cite[Theorem 3.1.1]{delaPena-Gine:99}) and symmetrization inequalities (\cite[Lemma 1.2.6]{delaPena-Gine:99} applied iteratively), we have
\begin{align*}
\Big\lVert \sum_{i_1, \ldots, i_d} a_{i_1 \ldots i_d} X_{i_1} \cdots X_{i_d} \Big\lVert_{L^p} &\le C_d \Big\lVert \sum_{i_1, \ldots, i_d} a_{i_1 \ldots i_d} X_{i_1}^{(1)} \cdots X_{i_d}^{(d)} \Big\lVert_{L^p}\\
&\le C_d \Big\lVert \sum_{i_1, \ldots, i_d} a_{i_1 \ldots i_d} \varepsilon_{i_1}^{(1)} X_{i_1}^{(1)} \cdots \varepsilon_{i_d}^{(d)}X_{i_d}^{(d)} \Big\lVert_{L^p}.
\end{align*}
By Kwapie\'{n}'s contraction principle \cite[Theorem 1]{kwapien:87} (more precisely, its decoupled analogue, which follows easily by iteration of Kwapie\'{n}'s result for linear forms conditionally on the other random variables), we furthermore obtain that
\[
\Big\lVert \sum_{i_1, \ldots, i_d} a_{i_1 \ldots i_d} \varepsilon_{i_1}^{(1)} X_{i_1}^{(1)} \cdots \varepsilon_{i_d}^{(d)}X_{i_d}^{(d)} \Big\lVert_{L^p} \le \Big\lVert \sum_{i_1, \ldots, i_d} a_{i_1 \ldots i_d} Y_{i_1}^{(1)} \cdots Y_{i_d}^{(d)} \Big\lVert_{L^p}
\]
for any $p \ge 1$, where $Y_1, \ldots, Y_n$ is a set of i.i.d.\ random variables with symmetrized Pareto distribution $X_i \sim \mathrm{Par}_\mathrm{s}(\alpha,1)$.

Let us now first assume $d=1$. By \cite[Theorem 1.1]{hitczenko-montgomery-smith-oleszkiewicz:97} applied to the symmetric log-convex random variables $a_1Y_1, \ldots, a_nY_n$, we obtain
\begin{align}
\Big\lVert \sum_{i=1}^n a_i Y_i \Big\lVert_{L^p} &\le C \Big(\Big(\sum_{i=1}^n |a_i|^p\lVert Y_i \lVert_{L^p}^p\Big)^{1/p} + \sqrt{p} \Big(\sum_{i=1}^n a_i^2\lVert Y_i \rVert_{L^2}^2\Big)^{1/2}\Big)\notag\\
&= C \Big(\Big(\sum_{i=1}^n |a_i|^p\Big)^{1/p} \Big(\frac{\alpha}{\alpha - p}\Big)^{1/p} + \sqrt{p} \Big(\sum_{i=1}^n a_i^2\Big)^{1/2} \Big(\frac{\alpha}{\alpha - 2}\Big)^{1/2}\Big)\label{intermediate}\\
&\le C_\alpha \Big(\sum_{i=1}^n a_i^2\Big)^{1/2} \Big(\frac{\alpha}{\alpha - p}\Big)^{1/p} = C_\alpha \lVert A \rVert_\mathrm{HS} \Big(\frac{\alpha}{\alpha - p}\Big)^{1/p}\notag
\end{align}
for any $p \in [2,\alpha)$. This may be iterated to arrive at
\begin{equation}\label{induction}
\Big\lVert \sum_{i_1, \ldots, i_d} a_{i_1 \ldots i_d} Y_{i_1}^{(1)} \cdots Y_{i_d}^{(d)} \Big\lVert_{L^p} \le C_\alpha^d \lVert A \rVert_\mathrm{HS} \Big(\frac{\alpha}{\alpha - p}\Big)^{d/p}
\end{equation}
To see this, assume that \eqref{induction} holds up to order $d-1$. It follows that
\begin{align*}
    &\qquad \Big\|\sum_{i_1, \ldots, i_d} a_{i_1 \ldots i_d}Y_{i_1}^{(1)}\cdots Y_{i_d}^{(d)}\Big\|_{L^p}^2\\
    &\le C_\alpha^{2(d-1)}\Big(\frac{\alpha}{\alpha - p}\Big)^{2(d-1)/p} \Big\|(\sum_{i_1, \ldots, i_{d-1}}(\sum_{i_d=1}^na_{i_1 \ldots i_d}Y_{i_d}^{(d)})^2)^{1/2}\Big\|_{L^p}^2\\
    &= C_\alpha^{2(d-1)}\Big(\frac{\alpha}{\alpha - p}\Big)^{2(d-1)/p} \Big\|\sum_{i_1, \ldots, i_{d-1}}(\sum_{i_d=1}^na_{i_1 \ldots i_d}Y_{i_d}^{(d)})^2\Big\|_{L^{p/2}}\allowdisplaybreaks\\
    &\le C_\alpha^{2(d-1)}\Big(\frac{\alpha}{\alpha - p}\Big)^{2(d-1)/p} \sum_{i_1, \ldots, i_{d-1}}\Big\|(\sum_{i_d=1}^na_{i_1 \ldots i_d}Y_{i_d}^{(d)})^2\Big\|_{L^{p/2}}\\
    &= C_\alpha^{2(d-1)}\Big(\frac{\alpha}{\alpha - p}\Big)^{2(d-1)/p} \sum_{i_1, \ldots, i_{d-1}}\Big\|\sum_{i_d=1}^na_{i_1 \ldots i_d}Y_{i_d}^{(d)}\Big\|_{L^p}^2\\
    &\le C_\alpha^{2d}\Big(\frac{\alpha}{\alpha - p}\Big)^{2d/p} \sum_{i_1, \ldots, i_{d-1}}\lVert a_{i_1 \ldots i_{d-1}\cdot}\rVert_{\ell^2}^2 = C_\alpha^{2d}\Big(\frac{\alpha}{\alpha - p}\Big)^{2d/p}\lVert A \rVert_\mathrm{HS}^2,
\end{align*}
using the induction in the first and the case of $d=1$ in the last inequality and denoting by $a_{i_1 \ldots i_{d-1}\cdot}$ the vector in $\mathbb{R}^n$ defined by fixing all the coordinates of the tensor $A$ but the last one (with Euclidean norm $\lVert \cdot \rVert_{\ell^2}$).
\end{proof}

Note that the $L^p$ estimates from \cite{hitczenko-montgomery-smith-oleszkiewicz:97} we used in the proof hold for $p \ge 2$ only, so that we must assume $\alpha > 2$ in Proposition \ref{TailsMultPol}. Similarly, one may think of replacing the assumption $\mathbb{P}(|X_i| \ge t) \le (b/t)^\alpha$ by a condition on the $L^p$ norms. However, in the proof we need stochastic dominance by a log-concave distribution. In particular, it is possible to reformulate Proposition \ref{TailsMultPol} with the Pareto-type tails replaced by the tails of the random variables introduced in Proposition \ref{prop-logcon}.

Let us compare the concentration bounds from Proposition \ref{TailsMultPol} to the behaviour of polynomial chaos in sub-Gaussian (and subexponential) random variables. In the latter cases, the tails get significantly heavier as $d$ increases. For instance, the tails of a $d$-th order chaos in sub-Gaussian random variables will be $2/d$-subexponential (i.\,e.\ of type $\exp(-t^{2/d})$) for large $t$. By contrast, the $d$-dependence in Proposition \ref{TailsMultPol} mostly appears in the logarithmic factor, while the (asymptotically dominating) Pareto-type tail decay still comes with shape parameter $\alpha$ irrespective of $d$.

Next, we show a version of Proposition \ref{TailsMultPol} which also includes diagonal terms. Note that if $X \sim \mathrm{Par}(\alpha,b)$ and $k \in \mathbb{N}$, then $X^k \sim \mathrm{Par}(\alpha/k,b^k)$. As a consequence, unlike in the sub-Gaussian or sub-exponential case again, the tail behaviour depends on the order of the largest power to which a single random variable $X_i$ is raised. More precisely, we may write $A = A_1 + \ldots + A_d$ with $d$-tensors $A_1, \ldots, A_d$ such that
\[
(A_k)_{i_1\ldots i_d} = \begin{cases}
a_{i_1\ldots i_d} & \text{if } k= \max\{\nu \colon \exists \, i_{j_1} = \ldots = i_{j_\nu}, j_1,\ldots, j_\nu \text{ pw.\ different}\}\\
0 & \text{else.}
\end{cases}
\]
For instance, if $A$ has generalized diagonal $0$ as in Proposition \ref{TailsMultPol}, then $A=A_1$.

Furthermore, if diagonal terms appear in a $d$-th order chaos, it is necessary to ``recenter'' it. This is already obvious for $d=2$, where the functional under consideration may be rewritten as
\[
\sum_{i \ne j} a_{ij} (X_i - \mathbb{E}[X_i])(X_j-\mathbb{E}[X_j]) + \sum_{i=1}^n a_{ii} (X_i^2 - \mathbb{E}[X_i^2]).
\]
Similarly, in higher orders, we need to replace each of the monomials $X_{i_1}^{k_1} \cdots X_{i_m}^{k_m}$ by $(X_{i_1}^{k_1} - \mathbb{E}[X_{i_1}^{k_1}]) \cdots (X_{i_m}^{k_m} - \mathbb{E}[X_{i_m}^{k_m}])$. As a consequence, given any $d$-tensor $A \in \mathbb{R}^{n^d}$, the natural generalization of \eqref{multPol} are functionals of the form
\begin{equation}\label{MultPolgen}
f_{d,A}(X) := \sum_{\nu=1}^d \sum_{\substack{k_1 \ge \ldots \ge k_\nu\ge 1\\k_1 + \ldots + k_\nu=d}} \sum_{i_1 \ne \ldots \ne i_\nu} \tilde{a}_{i_1\ldots i_\nu}^{k_1 \ldots k_\nu} (X_{i_1}^{k_1} - \mathbb{E}[X_{i_1}^{k_1}]) \cdots (X_{i_m}^{k_m} - \mathbb{E}[X_{i_m}^{k_m}]).
\end{equation}
Here, in terms of the tensor $A$, we have
\[
\tilde{a}_{i_1 \ldots i_\nu}^{k_1 \ldots k_\nu} = \sum a_{j_1\ldots j_d},
\]
where the sum extends over all $d$-tuples $j_1, \ldots, j_d$ in which every index $i_1, \ldots, i_\nu$ appears exactly $k_1,\ldots,k_\nu$ times. In particular, if $a_{j_1 \ldots j_d} \ne 0$, then it is a non-zero entry of $A_{k_1}$ (recall that $k_1 \ge \ldots \ge k_\nu$). In this sense, \eqref{MultPolgen} centers and ``regroups'' the indexes according to how many \emph{different} indexes appear in $i_1,\ldots,i_d$ and to which powers they are raised.

\begin{proposition}\label{TailsMultPolwdiag}
Let $X = (X_1, \ldots, X_n)$ be a random vector with independent centered components such that there exist $\alpha > 2$ and $b > 0$ with $\mathbb{P}(|X_i| \ge t) \le (b/t)^\alpha$ for any $t \ge b$ and any $i = 1, \ldots, n$. Let $A = (a_{i_1 \ldots i_d}) \in \mathbb{R}^{n^d}$ be a $d$-tensor and $f_{d,A}$ the functional defined in \eqref{MultPolgen}. Then,
\begin{align}
\lVert f_{d,A}(X) \lVert_{L^p} &\le C_{d,\alpha} \sum_{k=1}^{k^\ast} \lVert A_k \rVert_\mathrm{HS} b^d \Big(\frac{\alpha/k}{\alpha/k - p}\Big)^{(d-k+1)/p}\label{summom}\\
&\le C_{d,\alpha} \lVert A \rVert_\mathrm{HS} b^d \Big(\frac{\alpha/k^\ast}{\alpha/k^\ast - p}\Big)^{(d-k^\ast+1)/p}\notag
\end{align}
for any $p \in [2,\alpha/k^\ast)$, where $k^\ast := \max\{k \colon A_k \ne 0\}$. In particular,
\[
\mathbb{P}(|f_{d,A}(X)| \ge t) \le C_{d,\alpha} \max_{k=1,\ldots,k^\ast} \Big( \log^{d-k+1}\Big(\frac{t}{\lVert A_k \rVert_\mathrm{HS}b^d}\Big) \Big(\frac{\lVert A_k \rVert_\mathrm{HS}b^d}{t}\Big)^{\alpha/k}\Big)
\]
for all $t \ge C_{d,\alpha} b^d \max_{k=1,\ldots,k^\ast}\lVert A_k \rVert_\mathrm{HS}$.
\end{proposition}

Note that Proposition \ref{TailsMultPolwdiag} implicitly requires that $\alpha/k^\ast > 2$. To illustrate the result, let us consider the case of $d=2$. Here, $A = A_1 + A_2 = A^\mathrm{od} + A^\mathrm{d}$ for the off-diagonal and diagonal parts of $A$, so that the tail bound given in Proposition \ref{TailsMultPolwdiag} reads
\begin{align}\label{HWI}
\begin{split}
\mathbb{P}(|f_{2,A}(X)| \ge t) \le C_\alpha \max\Big(& \log^2\Big(\frac{t}{\lVert A^\mathrm{od} \rVert_\mathrm{HS}b^2}\Big) \Big(\frac{\lVert A^\mathrm{od} \rVert_\mathrm{HS}b^2}{t}\Big)^\alpha,\\ &\log\Big(\frac{t}{\lVert A^\mathrm{d} \rVert_\mathrm{HS}b^2}\Big)\Big(\frac{\lVert A^\mathrm{d} \rVert_\mathrm{HS}b^2}{t}\Big)^{\alpha/2}\Big).
\end{split}
\end{align}
While \eqref{HWI} is certainly not strong enough to be called a true analogue of the famous Hanson--Wright inequality for independent sub-Gaussian random variables, it can nevertheless be regarded as a preliminary result in this direction. We provide a more detailed comparison in the following subsection.

\begin{proof}
We again assume $b=1$. Note that $\mathbb{P}(|X_{i_j}^{k_j}| \ge t) \le t^{-\alpha/k_j}$ and therefore also $\mathbb{P}(|X_{i_j}^{k_j} - \mathbb{E}[X_{i_j}^{k_j}]| \ge t) \le C_{\alpha,k_j}t^{-\alpha/k_j}$ for $t$ sufficiently large (depending on $\alpha$ and $k_j$). Therefore, we may apply Proposition \ref{TailsMultPol} with $\nu$ instead of $d$ and $\alpha/k_1$ instead of $\alpha$ to each of the summands
\[
f^{k_1,\ldots,k_\nu}_{\nu,\tilde{A}}(X) := \sum_{i_1 \ne \ldots \ne i_\nu} \tilde{a}_{i_1 \ldots i_\nu}^{k_1 \ldots k_\nu} (X_{i_1}^{k_1} - \EW[X_{i_1}^{k_1}]) \cdots (X_{i_\nu}^{k_\nu}- \EW[X_{i_\nu}^{k_\nu}])
\]
from the representation \eqref{MultPolgen}, where we suppress the dependency of $\tilde{A}$ on $k_1,\ldots,k_\nu$ in our notation. Clearly, $\lVert \tilde{A} \rVert_\mathrm{HS} \le C_d \lVert A_{k_1} \rVert_\mathrm{HS}$ (recall that $k_1 = \max_jk_j$), and hence,
 \begin{align*}
\lVert f^{k_1,\ldots,k_\nu}_{\nu,\tilde{A}}(X) \lVert_{L^p} &\le C_{d,\alpha} \lVert \tilde{A} \rVert_\mathrm{HS} \Big(\frac{\alpha/k_1}{\alpha/k_1 - p}\Big)^{\nu/p}\\ &\le C_{d,\alpha} \lVert A_{k_1} \rVert_\mathrm{HS} \Big(\frac{\alpha/k_1}{\alpha/k_1 - p}\Big)^{(d-k_1+1)/p}.
 \end{align*}
Summing up according to $k_1 = k \in \{1,\ldots,k^\ast\}$ we arrive at \eqref{summom}. To derive the tail bound, use that $f_{d,A} = \sum_{k=1}^{k^\ast} f_{k,A_k}$, so that by union bound,
\[
\mathbb{P}(|f_{d,A}(X)| \ge t) \le \sum_{k=1}^{k^\ast} \mathbb{P}(|f_{k,A_k}(X)| \ge t/k^\ast).
\]
As in Proposition \ref{TailsMultPol}, the claim now follows from the $L^p$ estimates we derived above (for each $k$) in combination with \eqref{eq:log-error-tails-gen}.
\end{proof}

\subsection{Discussion and open questions}

For general $d \ge 2$, the results from Propositions \ref{TailsMultPol} and \ref{TailsMultPolwdiag} seem new. However, alternate bounds for linear forms, i.\,e., the case of $d=1$, can be deduced from existing literature. In particular, we shall consider the Fuk--Nagaev inequality which is due to \cite{fuk-nagaev:1971}, cf.\ also \cite{fuk:1973} and \cite{nagaev:1979} (for a more recent partial refinement, see \cite{chesneau:09}). Given a sequence $X_1,\ldots,X_n$ of independent random variables with $\mathbb{E}[X_i] = 0$ and finite second moments, let $S_n := X_1 + \ldots + X_n$, $\sigma^2 := \mathbb{E}[X_1^2] + \ldots + \mathbb{E}[X_n^2]$, and
\[
C_p(X) := \Big(\sum_{i=1}^n \mathbb{E}[\max(0,X_i)^p]\Big)^{1/p}
\]
for any $p \ge 1$. Then, by \cite[Corollary 1.8]{nagaev:1979},
\[
\mathbb{P}(S_n > t) \le \Big(\frac{(p+2)C_p(X)}{pt}\Big)^p + \exp\Big(-\frac{2t^2}{(p+2)^2\eul^p\sigma^2}\Big)
\]
for any $t > 0$.

Now assume that $X_1', \ldots, X_n'$ have $L^p$ norms at most of Pareto type \eqref{Paretomoments}, say, with $b=1$, and let $X_1 := a_1X_1', \ldots, X_n := \, a_nX_n'$ for real-valued coefficients $a_1, \ldots, a_n$. Writing $a=(a_1,\ldots,a_n)$ and $\lVert a \rVert_{\ell^p}$ for its $\ell^p$ norm, we clearly have
\[
C_p(X)^p \le \lVert a \rVert_{\ell^p}^p \frac{\alpha}{\alpha-p},\qquad \sigma^2 \le \lVert a \rVert_{\ell^2}^2 \frac{\alpha}{\alpha-2}.
\]
It follows that in this case,
\[
\mathbb{P}(S_n > t) \le \frac{(p+2)^p}{p^p} \frac{\alpha}{\alpha-p} \frac{\lVert a \rVert_{\ell^p}^p}{t^p} + \exp\Big(-\frac{2(\alpha-2)t^2}{(p+2)^2\eul^p\lVert a \rVert_{\ell^2}^2\alpha}\Big)
\]
for any $p \in (2,\alpha)$ and any $t > 0$. One may now plug in $p=\alpha - 1/\log(t/\lVert a \rVert_{\ell^\alpha})$ to arrive at a bound which scales with $1/t^\alpha$ as in \eqref{eq:log-error-tails}, though in this case we additionally have to take care of the $\ell^p$-norms $\lVert a \rVert_{\ell^p}$ (e.\,g.\ by upper-bounding them). We therefore refrain from presenting the resulting inequality. The same holds for the lower tails $\PP(S_n \le - t)$. The resulting (two-sided) inequality is obviously sharper than the bound given by Proposition \ref{TailsMultPol} for $d=1$ and $b=1$, which reads
\[
\mathbb{P}\Big(\Big|\sum_{i=1}^n a_iX_i\Big| \ge t\Big) \le C_\alpha \log(t/\lVert a \rVert_{\ell^2}) \Big(\frac{\lVert a \rVert_{\ell^2}}{t}\Big)^\alpha
\]
for all $t \ge C_\alpha \lVert a \rVert_{\ell^2}$.

In fact, this result corresponds to what we heuristically get from \eqref{intermediate} if we naively optimize each of the two summands of \eqref{intermediate} on its own, recalling \eqref{Paretomoments} and standard results about sub-Gaussian tail behavior like \cite[Proposition 2.5.2]{vershynin:2018} (in particularly pretending we may consider the second summand for any $p \in [2,\infty)$). However, this is an informal argument, and in practice we can only address the ``sub-Gaussian'' part for $p \in [2,\alpha)$, where it is almost meaningless (one needs to access $p \to \infty$ to obtain sub-Gaussian tail bounds). Yet, on a heuristical level the $L^p$ bounds do seem to imply the Fuk--Nagaev results, and it is an interesting question whether this observation can be made rigorous.

Similar remarks hold for $d\ge2$. Here, the aforementioned results for Gaussian chaos or polynomials in sub-Gaussian or sub-exponential random variables identify a wealth of different tail levels which scale with different tensor-type norms of the coefficient tensor $A$ (or more generally of the derivatives of the polynomial $f$ under consideration). By contrast, Propositions \ref{TailsMultPol} and \ref{TailsMultPolwdiag} only involve much fewer levels of tail decay and always the Hilbert--Schmidt norm, which is the largest norm among the family of norms which appears in those results.

For instance, in the case of $d=2$ the by now classical Hanson--Wright inequality (cf.\ e.\,g.\ \cite{hanson-wright:1971,wright:1973,rudelson-vershynin:2013}) states that if the $X_1, \ldots, X_n$ are independent sub-Gaussian random variables with mean $0$ and variance $1$,
\[
\mathbb{P}(|f_{2,A}(X) - \mathbb{E}[f_{2,A}(X)]| \ge t) \le 2 \exp\Big(-\frac{1}{C} \min\Big( \Big(\frac{t}{\lVert A \rVert_\mathrm{HS}}\Big)^2, \frac{t}{\lVert A \rVert_\mathrm{op}}\Big)\Big),
\]
where $\lVert A \rVert_\mathrm{op}$ is the operator norm of $A$ and $C > 0$ is some constant which only depends on the sub-Gaussian norms of the $X_i$. Analogues of the Hanson-Wright inequality for sub-exponential random vectors can be found in \cite{goetze-sambale-sinulis:2021, sambale:2023}. While identifying two regimes of tail decay, the ``heavy-tailed'' Hanson--Wright type inequality \eqref{HWI} does not yet seem to be optimal in the sense of involving different norms (of the matrix $A$) and scaling regimes. One may invoke a sub-Gaussian part by treating the diagonal part with the Fuk--Nagaev result instead of Proposition \ref{TailsMultPol}, which yields the bound
\begin{align}\label{HWI+FNI}
\begin{split}
&\mathbb{P}(|f_{2,A}(X)| \ge t) \le C_\alpha\max\Big(\exp\Big(-C_\alpha'\Big(\frac{t}{\lVert a^{\mathrm{d}} \rVert_{\ell^2}b^2}\Big)^2\Big),\\
&\qquad\frac{\alpha/2}{\alpha/2-p} \frac{b^p\lVert a^{\mathrm{d}} \rVert_{\ell^p}^p}{t^p}, \log^2\Big(\frac{t}{\lVert A^\mathrm{od} \rVert_\mathrm{HS}b^2}\Big)\Big(\frac{\lVert A^\mathrm{od} \rVert_\mathrm{HS}b^2}{t}\Big)^\alpha\Big)
\end{split}
\end{align}
for any $p < \alpha/2$, where $a^\mathrm{d}$ is the vector consisting of the diagonal entries of $A$. However, even if \eqref{HWI+FNI} improves upon \eqref{HWI} it should still be non-optimal.

In principle, an obvious strategy would be to sharpen the proof of Proposition \cite{hitczenko-montgomery-smith-oleszkiewicz:97} by the results derived in \cite{kolesko-latala:2015} (similar to the methods used in \cite{goetze-sambale-sinulis:2021}). However, even if we possibly arrive at $L^p$ bounds of the form
\[
\lVert f_{d,A}(X) \rVert_{L^p} \le \sum_i C(i,\alpha,A) \varphi_{i,\alpha}(p)
\]
for suitable coefficients $C(i,\alpha,A)$ and $p$-dependent functions $\varphi_{i,\alpha}(p)$, it is again hard to derive multilevel tail inequalities from them for similar reasons as sketched in the case of $d=1$. Indeed, in the case of exponential-type tails, $\varphi(p) \approx p^\kappa$ for suitable exponents $\kappa > 0$, while in the heavy-tailed situation $\varphi(p) \approx (\kappa/(\kappa - p))^{1/p}$ for $\kappa \approx \alpha/i$. In particular, these inequalities will only hold for $p \in [p,\alpha/d)$, so that we can only access $p$ in a region where most of the $L^p$ bounds are basically meaningless.

Note in passing that for this reason, the proof of the tail bound in Proposition \ref{TailsMultPolwdiag} involves a union bound (thus splitting the functional into several parts), which is not necessary in the case of sub-exponential random variables. On the other hand, slightly generalizing the case of $d=2$ in Proposition \ref{TailsMultPolwdiag}, observe that if $X$ is a positive random variable such that
    \begin{equation}\label{Tails}
    P(X \ge x) = \max \Big(\frac{b^{2\alpha}}{x^{2\alpha}}, \frac{a^\alpha}{x^\alpha}\Big)
    \end{equation}
    for all $x \ge b$, where $0 < a < b< \infty$ and $\alpha > 0$, we have
    \begin{equation}\label{Moments}
    \mathbb{E}[X^p] \le b^p \frac{2\alpha}{2\alpha-p} + a^p \frac{\alpha}{\alpha-p}.
    \end{equation}
Indeed, an easy calculation yields
\begin{align*}
        \mathbb{E}[X^p]
        &= p \int_0^b x^{p-1}dx + p \int_b^{b^2/a} x^{p-1} \frac{b^{2\alpha}}{x^{2\alpha}} dx + p \int_{b^2/a}^\infty x^{p-1} \frac{a^\alpha}{x^\alpha} dx\\
        &= b^p + \frac{pb^{2\alpha}}{p-2\alpha}\Big(\frac{b^{2p-4\alpha}}{a^{p-2\alpha}} - b^{p-2\alpha}\Big) - \frac{pa^\alpha}{p-\alpha} \frac{b^{2p-2\alpha}}{a^{p-\alpha}}\\
        &= b^p \frac{2\alpha}{2\alpha-p} + a^p \frac{b^{2p-2\alpha}}{a^{2p-2\alpha}} \frac{p}{2\alpha-p} \frac{\alpha}{\alpha-p}
        \le b^p \frac{2\alpha}{2\alpha-p} + a^p \frac{\alpha}{\alpha-p}.
    \end{align*}
    However, it is not clear whether \eqref{Moments} also implies \eqref{Tails} (maybe up to a logarithmic error), which could serve as a starting point for more precise tail bounds for chaos in heavy-tailed random variables.

\end{document}